\documentclass{amsart}

\usepackage{amsmath, amssymb, amsthm, mathtools, mathrsfs, hyperref}
\usepackage{cleveref}
\usepackage{stmaryrd}
\usepackage{enumitem}

\usepackage{todonotes}

\usepackage[sc]{mathpazo}
\usepackage[T1]{fontenc}

\usepackage[margin = 1.35 in]{geometry}

\usepackage{multicol}

\usepackage{tikz}
\usetikzlibrary{shapes, arrows.meta, positioning}

\usepackage{xparse}

\usepackage{cite}

\usepackage{xcolor}

\newtheorem{theorem}{Theorem}[section]
\newtheorem{lemma}[theorem]{Lemma}

\newtheorem{proposition}[theorem]{Proposition}

\newtheorem{problem}[theorem]{Problem}

\theoremstyle{definition}
\newtheorem{definition}[theorem]{Definition}

\newtheorem{setting}[theorem]{Setting}
\newtheorem{remark}[theorem]{Remark}

\DeclareMathOperator{\Pic}{Pic}

\DeclareMathOperator{\Ram}{Ram}

\DeclareMathOperator{\length}{length}

\DeclareMathOperator{\rk}{rk}

\renewcommand{\P}{\mathbb{P}}

\newcommand{\G}{\mathbb{G}}
\newcommand{\Z}{\mathbb{Z}}

\renewcommand{\phi}{\varphi}
\newcommand{\wt}[1]{\widetilde{#1}}
\renewcommand{\bar}[1]{\overline{#1}}

\title{Bounding Pinch Point Schemes of Projected Surfaces}

\author{Adam Cartisano, Anand Patel}

\begin{document}

\begin{abstract}
	Let $X$ be a smooth surface and let $\phi:X\to\P^N$, with $N\geq 4$, be a finitely ramified map which is birational onto its image $Y = \phi(X)$, with $Y$ non-degenerate in $\P^N$. In this paper, we produce a lower bound for the length of the pinch point scheme of a general linear projection of $Y$ to $\P^3.$ We then prove that the lower bound is realized if and only if $Y$ is a rational normal scroll.
\end{abstract}

\maketitle

\section{Introduction}

The classical General Projection Theorem (see \cite[Theorem 2.5]{ciliberto2011branch}, \cite[Section 2]{piene2005singularities},\cite[Section 1]{piene1978some},\cite[Theorem 1]{roberts1971generic}) states that there are at most three types of singularities (often called \textit{ordinary singularities}) contained in the image of a general linear projection of a smooth surface to $\P^3$. These include: A curve of double points, finitely many triple points, and finitely many pinch points. It is natural to wonder what types of numerical limitations these singularities admit.

Let $X$ be a smooth projective surface in $\P^N$. In this paper, we focus on the third type of singularity, pinch points. In particular, we present a lower bound for the length of the \textit{pinch point scheme} of a general projection of $X$ to $\P^3$. The length of this scheme describes the geometry of $X$ as it sits inside $\P^N$. While there is clearly no upper bound on the pinch point scheme length, we will demonstrate a lower bound in terms of $N$.

\begin{theorem}
    \label{theorem:simplified} Let $X \subset \P^{N}$ be a non-degenerate, smooth projective surface with $N \geq 4$. Then the number of pinch points of a general projection of $X$ to $\P^{3}$ is at least $2N-6$, with equality holding if and only if $X$ is a rational normal scroll.
\end{theorem}

The numerical invariant studied in this paper is naturally a projective character. If $X\subset\P^N$ is a smooth embedded surface, its Gauss map $\mathfrak g:X\to\G(2,N)$ determines a cycle class
$$\mathfrak g_*[X]=\gamma_2\sigma_2^*+\gamma_{1,1}\sigma_{1,1}^*.$$
The coefficient $\gamma_{1,1}$ is the class of $X$, equivalently the degree of its dual variety (when the dual is a hypersurface), while $\gamma_2$ is the length of the pinch-point scheme arising from a general projection of $X$ to $\P^3$. Thus Theorem \ref{theorem:simplified} gives the sharp lower bound $\gamma_2\geq2N-6,$ with equality precisely for rational normal scrolls. In the proof, we work more generally with maps $\phi:X\to\P^N$ that need not be embeddings. In that setting, we continue to write $\gamma_2$ for the corresponding Chern-class expression, which still computes the length of the ramification scheme of a general projection. 

\begin{remark} We make a few remarks to properly contextualize the theorem.  
\begin{enumerate}
    \item The analogous statement for curves is an immediate consequence of the Riemann--Hurwitz formula. Let $X\subset\P^N$ be a non-degenerate smooth curve of degree $d$ and genus $g(X)$. A general projection $X\to\P^1$ has degree $d$, and its ramification divisor has degree $2d+2g(X)-2$. Since $d\geq N$ and $g(X)\geq0$, it has at least $2N-2$ ramification points, counted with multiplicity, with equality if and only if $X\subset\P^N$ is a rational normal curve of degree $N$.
    
    \item Theorem \ref{theorem:simplified} should be compared with a result of E. Marchionna \cite{marchionna1955sopra} from the 1950s: if $X\subset\P^N$, with $N\geq4$, is a non-degenerate linearly normal smooth surface of degree $d$, then the number of singular elements in a general pencil of hyperplane sections of $X$ is at least $d-1$, with equality if and only if $X$ is isomorphic to the Veronese surface in $\P^5$. Hence, for $N\geq5$, this number is at least $N-2$, with equality if and only if $N=5$ and $X$ is isomorphic to the Veronese surface. Moreover, this number is $N-1$ if and only if $X\subset\P^N$ is a smooth rational normal scroll. The methods of our proof are entirely different from those used by Marchionna.

    \item The character $\gamma_2$ admits two descriptions, one via Chern classes and the other via the tangent variety. By \cite[Proposition 12.6]{eisenbud20163264},
    $$\gamma_2=\deg\left(6\zeta^2+4\zeta K_X+K_X^2-c_2(T_X)\right),$$
    a formula from which the lower bound $2N-6$ is not apparent. Let $\Theta_X\subset\P^N$ be the tangent variety and let $\eta$ be the natural map from the abstract tangential variety onto $\Theta_X$. If $\dim\Theta_X=4$, then $\eta$ is generically finite and  $\gamma_2=\deg(\eta)\deg\Theta_X.$ Here $\deg(\eta)$ is the tangent degree. Hernandez Gomez and Russo study both factors in \cite{gomez2026tangentdegreedegreetangent} and prove, when $\Theta_X\subsetneq\P^N$, that $\deg\Theta_X\geq2(N-\dim\Theta_X+1).$
\end{enumerate}
\end{remark}

In this paper, we prove the stronger Theorem \ref{thm:ppbound}, in which the embedded surface $X\subset\P^N$ is replaced by an arbitrary uncrumpled map $\phi:X\to\P^N$; see Definition \ref{def:uncrumpled}. This extends the usual setting of maps obtained from embedded surfaces by linear projection. The broader formulation is also well adapted to our inductive use of general inner projections.

The rest of the paper will proceed as follows. In Section \ref{sec:defncon}, we list out the specific definitions, hypotheses and conventions used throughout the rest of the paper, along with some preliminary results needed later on. Section \ref{sec:mainproof} contains the proof for Theorem \ref{thm:ppbound}, preceded by a series of constituent lemmas. In Section \ref{sec:generalization}, we generalize our result to classify, for $N$ large enough, the surfaces which have a near-minimal pinch point scheme length. Finally, in Section \ref{sec:futurework} we discuss future directions. Our work points to many natural problems in the geometry of projective varieties, and we hope interested readers will pursue them in the future. 

\section{Definitions and Conventions}
\label{sec:defncon}
Throughout what follows, we work over an algebraically closed field $k$ of characteristic $0$. The word \textit{variety} will be taken to mean a reduced finite type $k$-scheme, and the words \textit{curve} and \textit{surface} will refer to varieties of dimension one and two respectively. We adopt the pre-Grothendieck convention for projective spaces, i.e. $\P_k^N$ will refer to the set of one-dimensional subspaces of $k^{N+1}$.

Let $X$ be an irreducible smooth surface. For a map $\phi:X\to \P^N$, we write $d\phi:T_X\to \phi^*T_{\P^N}$ for the map on tangent bundles induced by $\phi$, and we denote by $\Ram(\phi)$ the ramification scheme of $\phi$, i.e. the subscheme of $X$ for which $\rk(d\phi)\leq 1.$

\begin{definition}
    \label{def:uncrumpled}
    For a smooth projective surface $X$, we call a map $\phi:X\to\P^N$ \textbf{uncrumpled} if 
    \begin{enumerate}
        \item $\phi(X)\subset\P^N$ is non-degenerate, i.e. not contained in a hyperplane,
        \item $X$ is birational onto its image $\phi(X)$, and
        \item $\Ram(\phi)$ is finite.
    \end{enumerate}
\end{definition}

Observe in particular that an uncrumpled map $\phi$ is necessarily finite onto its image $\phi(X)$. Since $X$ is smooth, it follows that $X$ is the normalization of $\phi(X)$.

\begin{definition}
    A surface $X\subset\P^N$ is \textbf{ruled by lines} if its Fano scheme is positive-dimensional.
\end{definition}

\begin{setting}[The Inner Projection Setting]
\label{innerprojectionsetting}
Let $X$ be an irreducible smooth surface, and let ${\phi:X\to\P^N}$ be an uncrumpled map whose image we denote $Y = \phi(X)$, and suppose $N \geq 4$. Let $x\in X$ be a generally chosen point, and put $y = \phi(x)$ (so that $y$ is a general point of $Y$). Let $\beta_x:\wt{X}\to X$ be the blow-up of $X$ at $x$, let $E = \beta_x^{-1}(x)$ be the exceptional curve, and let $\beta_y:\wt{Y}\to Y$ the blow-up of $Y$ at $y$. Next, let $\pi_y:Y\dashrightarrow\P^{N-1}$ be the \textit{inner projection from $y$}, i.e. the effect of linear projection from the point $y$. Let $\pi_x = \pi_y\circ\phi:X\dashrightarrow\P^{N-1}$ be the \textit{inner projection from $x$.} Let $\wt{\phi}:\wt{X}\to\P^{N-1}$ and $\wt{\pi_y}:\wt{Y}\to\P^{N-1}$ be the resolutions of $\pi_x$ and $\pi_y$ to the corresponding blow-ups. We use $\zeta$ in the Chow ring $A(X)$ and $\wt{\zeta}\in A(\wt{X})$ to denote the pullback of the hyperplane class along $\phi$ and $\wt{\phi}$ respectively and set $d = \deg Y = \zeta^2$. Lastly, we define the \textit{pinch point scheme length} of $\phi$ to be the quantity given by the intersection formula as in \cite{eisenbud20163264}
\begin{equation}
\label{equation:pinch}
\gamma_2 = \deg\left(6\zeta^2 + 4\zeta K_X + K_X^2 - c_2(T_X)\right).
\end{equation}
\end{setting}
We write $\wt\gamma_2$ for the corresponding invariant of the resolved inner projection $\wt\phi$.

Setting \ref{innerprojectionsetting} can be summarized in the following diagram:

\begin{center}
\begin{tikzpicture}[>=stealth]
    \node (Xt) at (0,1.75) {$\wt X$};
    \node (Yt) at (3,1.75) {$\wt Y$};
    \node (X)  at (0,0) {$X$};
    \node (Y)  at (3,0) {$Y$};
    \node (P)  at (1.5,-1.8) {$\P^{N-1}$};

    \draw[->] (Xt) -- node[right] {$\beta_x$} (X);
    \draw[->] (Yt) -- node[left] {$\beta_y$} (Y);
    \draw[->] (X) -- node[above] {$\phi$} (Y);
    \draw[->,dashed] (X) -- node[below left] {$\pi_x$} (P);
    \draw[->,dashed] (Y) -- node[below right] {$\pi_y$} (P);
    \draw[->] (Xt.west) to[bend right=65] node[left] {$\wt\phi$} (P.west);
    \draw[->] (Yt.east) to[bend left=65] node[right] {$\wt\pi_y$} (P.east);
\end{tikzpicture}
\end{center}

\begin{remark}
Let $\pi:Y\to\P^3$ be a general linear projection and put $f=\pi\circ\phi$. Then $\Ram(f)=D_1(df)$ is zero-dimensional, hence has the expected codimension $2$. By the Thom--Porteous formula, $[\Ram(f)]=c_2(f^*T_{\P^3}-T_X),$ whose degree is the expression in Equation \eqref{equation:pinch}. Thus $\gamma_2=\length\Ram(f)$. If $\phi$ is an embedding, the General Projection Theorem identifies the reduced points of this scheme with the pinch points of $\pi(Y)$.
\end{remark}

With the notion of an uncrumpled map in our hands, we can state the more general version of the main theorem.

\begin{theorem}
\label{thm:ppbound}
Let $X$ be a smooth surface, and let $\phi:X\to\P^N$ be an uncrumpled map to $Y = \phi(X)$ with $N\geq 4$. Then for a general linear projection $\pi:Y\to\P^3$, the length of the pinch point scheme of the composition $X \to \P^{3}$ is bounded below by $2N-6$ with equality holding if and only if $Y$ is a rational normal scroll.
\end{theorem}

We conclude this section  by collecting several classical results needed in the proof. 

\begin{theorem}[General Position Theorem \cite{harris1979galois}, \cite{arbarello2011geometry}]
\label{thm:gpt}
Let $C\subset \P^r$,  $r \geq 2$, be an irreducible non-degenerate, possibly singular, curve of degree $d$. Then a general hyperplane meets $C$ in $d$ points, any $r$ of which are linearly independent.
\end{theorem}

\begin{theorem}[Surfaces with a two-dimensional family of plane curves \cite{mezzetti1997tour}, \cite{sierra2007some}, \cite{segre1921superficie}]
\label{thm:toomanycurves}
The only surfaces $S \subset \P^N$, $N > 3$, containing a 2-dimensional family
of plane curves are: (1) the Veronese surface in $\P^5$, (2) any external projection of the Veronese surface to $\P^4$, (3) the smooth rational normal scroll
in $\P^4$, and (4) the cones.
\end{theorem}

\begin{theorem}[Linearity of general Gauss fibers {\cite[I.2.3(c)]{zak1993tangents}}]
\label{thm:gaussfibers}
In characteristic $0$, the closure of a general fiber of the Gauss map of an irreducible projective variety is a linear space.
\end{theorem}

\begin{theorem}[A Consequence of Fulton-Hansen Connectedness {\cite[Theorem 3.4.1]{lazarsfeld2017positivity}}]
\label{thm:fulhanconsequence}
Let $X$ be a complete irreducible variety of dimension $n$, and let $f:X\to\P^r$ be an unramified morphism. If $2n > r$, then $f$ is a closed embedding.
\end{theorem}

\section{Proof of Theorem \ref{thm:ppbound}}
\label{sec:mainproof}

In this section, we maintain the conventions established in Section \ref{sec:defncon}, and the context of Setting \ref{innerprojectionsetting}. As we proceed to the proof of Theorem \ref{thm:ppbound}, we shall require a series of lemmas, which we establish presently. Our first pair of lemmas serve as observations about the inner projection construction, where we address the corresponding properties of uncrumpled maps. Note that Lemma \ref{lem:nondegeneracy} is stated without proof.

\begin{lemma}
    \label{lem:nondegeneracy}
    The image surface $\wt{\pi_y}(\wt{Y})\subset\P^{N-1}$ is again non-degenerate.
\end{lemma}


\begin{lemma}
    \label{lem:birationality}
    The map $\wt{\pi_y}:\wt{Y}\to\P^{N-1}$ is birational onto its image $\wt{\pi_y}(\wt{Y})$.
\end{lemma}
    
\begin{proof}  
If $\dim\wt{\pi_y}(\wt{Y}) = 1$, then $Y$ is a cone with vertex $y$. Since this is true for a general $y\in Y$, then the line joining any two points on $Y$ is contained in $Y$, implying that $Y$ is an impossible 2-plane. As $\dim \wt{\pi_y}(\wt{Y}) = 0$ is clearly impossible, it follows that $\wt{\pi_y}(\wt{Y})$ must be $2$-dimensional.

Suppose $\deg\wt{\pi_y}>1$. Then a general secant line of $Y$ is trisecant. For a general codimension-two linear space $\Lambda\subset\P^N$, Theorem \ref{thm:gpt}, applied to a general hyperplane section of $Y$, says that $Y\cap\Lambda$ is in linear general position. Yet a general pair of its points lies on a trisecant contained in $\Lambda$, a contradiction.
\end{proof}

We wish to investigate the possibility that the map $\wt{\phi}$ is ramified at every point of the exceptional curve $E$. To that end, we prove a significantly more general result:

\begin{proposition}
    \label{prop:ramexceptional}
    Let $X$ be any surface, and let $\phi:X\to\P^N$ be birational onto its image $Y = \phi(X)$. If the inner projection map $\wt{\phi}:\wt{X}\to\P^{N-1}$ from a general point $y = \phi(x) \in Y$ is ramified along the exceptional curve $E\subset\wt{X}$, then $Y$ is a 2-plane.
\end{proposition}

\begin{proof}
Since $x$ is chosen generally and $\phi$ is birational onto its image, we may assume that $x$ and $y$ are smooth points and that $\phi$ is an isomorphism in a neighborhood of $x$. Thus, locally near the exceptional curve $E$, the map $\wt{\phi}$ may be identified with the resolution $\wt{\pi_y}:\wt Y\to \P^{N-1}$.

The ramification $\Ram \wt\phi$ along $E=\P(T_yY)$ is measured by the second fundamental form of $Y$ at $y$. More precisely, for any tangent direction $[v]\in \P(T_yY)$ with corresponding point $\wt{[v]}\in E$, we have 
$$\rk d\wt\phi(\wt{[v]}) <2 \iff \operatorname{II}_{Y,y}(v,v)=0.$$

Since $\operatorname{II}_{Y,y}$ is symmetric and the characteristic is not $2$, the vanishing of $\operatorname{II}_{Y,y}(v,v)$ for every $v\in T_yY$ implies that $\operatorname{II}_{Y,y}=0$. As $y$ was chosen generally, the second fundamental form therefore vanishes at a general point. Equivalently, the differential of the Gauss map $\mathfrak g:Y\dashrightarrow \G(2,N)$ vanishes generically, so the Gauss map is constant on a dense open subset of $Y$. Hence the embedded tangent plane to $Y$ is fixed along a dense open subset. Since every smooth point of $Y$ lies in its embedded tangent plane, this dense open subset is contained in a fixed 2-plane. Taking the closure, $Y$ is a 2-plane.
\end{proof}

In the next proposition, we address the third property of uncrumpled maps, concerning whether $\wt{\phi}$ is finitely ramified.

\begin{proposition}
    \label{prop:finiteramification}
    Either $\wt{\phi}:\wt{X}\to\P^{N-1}$ is finitely ramified, or else $Y$ is ruled by lines.
\end{proposition}

\begin{proof} We begin by asserting that the ramification locus $\Ram(\wt{\phi})$ is properly contained inside $\wt{X}$. Indeed, if $\dim \Ram(\wt{\phi})=2$, then $\Ram (\wt{\phi})=\wt{X}$, since the ramification scheme of a map is a closed subscheme of the domain. Then $Y$ is a cone over a curve with vertex $\phi(x)$. But since $x$ is chosen generally, $Y$ is a degenerate 2-plane.

Suppose instead that $\wt{\phi}$ is ramified along a curve $R\subset\wt{X}$, not containing $E$ by Proposition \ref{prop:ramexceptional}. Then $\beta_x(R)\subset X$ is a curve. For any $p\in X\setminus\Ram(\phi)$, we define $\Lambda_p \subset \P^{N}$ to be the projective 2-plane $$\Lambda_p := \langle d\phi(T_pX)\rangle\subset \P^N,$$
the linear span of the Zariski tangent plane at $\phi(p)$. We also define the incidence variety $\Sigma$ as the closure of the set $$\left\{(p,q)\mid p,q\notin\text{Ram}(\phi),\; \text{and } \phi(p)\in\Lambda_q\setminus\{\phi(q)\}\right\}\subset X\times X.$$
Let $\pi_1,\pi_2:\Sigma\to X$ be the projection maps to the first and second factor respectively. A generic fiber of $\pi_1$ is one-dimensional; this is a restatement of the hypothesis that $\wt{\phi}$ is ramified along the curve $R$. By dimension counting, it follows that $\dim \Sigma = 3.$ 

But now consider the map $\pi_2$. The fiber $\pi_2^{-1}(q)$ over a point $q\in X$ is the set of points in $X$ whose image is contained in a line tangent to $Y$ at $\phi(q)$. Equivalently, $$\pi_2^{-1}(q) = \phi^{-1}(\Lambda_q\cap Y).$$ 
Again by dimension counting, we see that $\dim \left(\pi_2(\Sigma)\right)\geq 1$. If $\pi_2(\Sigma)$ is a curve $C\subset X$, then the fiber over a general point $c\in C$ will be two-dimensional. Since $\Ram(\phi)$ is finite, $\Lambda_c$ is a well-defined 2-plane. But if $\pi_2^{-1}(c)$ is two-dimensional, then $Y\subset\Lambda_c$, so $Y$ is again degenerate.

Suppose now that $\pi_2$ is dominant. A general fiber determines a plane-curve component $C_q\subset\Lambda_q\cap Y$, and these curves cover $Y$. Let $\mathscr F$ denote the resulting irreducible family of plane curves. If $\dim\mathscr F\geq2$, Theorem \ref{thm:toomanycurves} gives the Veronese surface, its external projection, a cubic scroll, or a cone. The inner projection of the Veronese surface is the cubic-scroll embedding. For an external projection from $o$, it is the projection of $v_2(\P^2)$ from $\ell=\langle o,v_2(x)\rangle$; since $\ell\not\subset\operatorname{Tan}(v_2(\P^2))$ and a rank-two conic has only finitely many linear factors, only finitely many tangent planes meet $\ell$. Thus this map is finitely ramified. The remaining cases are ruled by lines.
    
We must now investigate the possibility that $\dim \mathscr{F} = 1$. Then for a general point $q\in X$, there are infinitely many points in $X$ whose fiber under $\pi_2$ is $C_q$. Let $r\in X$ be general among such points. Then $C_q\subset\Lambda_q\cap\Lambda_r.$ If the planes $\Lambda_q$ and $\Lambda_r$ are distinct, then $C_q$ is their line of intersection, implying $Y$ is ruled by lines. If $\Lambda_q=\Lambda_r$, then the Gauss map of $Y$ has positive-dimensional general fibers. By Theorem \ref{thm:gaussfibers}, their closures are linear spaces. The general fiber cannot have dimension $2$, so their closures are lines covering $Y$. Hence $Y$ is ruled.
\end{proof}

Thus, from Lemmas \ref{lem:nondegeneracy} and \ref{lem:birationality} and Proposition \ref{prop:finiteramification} we deduce the following key fact which will be used in the proof of Theorem \ref{thm:ppbound}: 
\begin{theorem} 
\label{thm:uncrumpled}
If $\phi:X\to\P^N$ with $N\geq 4$ is an uncrumpled map on a smooth surface $X$ with $\phi(X)$ not ruled by lines, then the general inner projection $\wt{\phi}:\wt{X}\to\P^{N-1}$ is uncrumpled.
\end{theorem}

\begin{lemma}
    \label{lem:inductivestep}
    If $Y$ is not ruled by lines, then $\wt\gamma_2 = \gamma_2 - 4.$
\end{lemma}
    
\begin{proof} Since $Y$ is not ruled, $\wt{\phi}$ is uncrumpled by Theorem \ref{thm:uncrumpled}. To keep the notation manageable, we write $E = [E]$, $\beta = \beta_x$, and we omit the degree map on zero cycles. Recall the following familiar properties from intersection theory on the blow-up:
\begin{multicols}{3}
\begin{itemize}
    \item $ (\beta^*\zeta)^2 = \zeta^2$
    \item $ (\beta^*K_{X})^2 =  K_X^2$
    \item $(\beta^*\zeta)(\beta^*K_{X}) = \zeta\cdot K_X$
    \item $(\beta^*\zeta)\cdot E = 0$
    \item $(\beta^*K_X)\cdot E = 0$
    \item $E^2 = -1$
    \item $\wt{\zeta} = \beta^*\zeta - E$
    \item $K_{\wt{X}} = \beta^*K_X + E$
    \item $c_2(T_{\wt{X}}) = c_2(T_X) + 1$. 
\end{itemize}
\end{multicols}

Applying Equation \eqref{equation:pinch} to $\wt{\phi}$, we execute a straightforward computation:
\begin{align*}
    \wt\gamma_2 & = 6\wt{\zeta}^2 + 4\wt{\zeta}\cdot K_{\wt{X}} + K_{\wt{X}}^2 - c_2(T_{\wt{X}})\\[5 pt]
    & = 6(\beta^*\zeta - E)^2 + 4(\beta^*\zeta - E)\cdot(\beta^*K_X + E) + (\beta^*K_X + E)^2 - (c_2(T_X) + 1)\\[5 pt]
    & = \left(6\zeta^2 + 4\zeta\cdot K_X + K_X^2 - c_2(T_X)\right) - 4\\[5 pt]
    & = \gamma_2 - 4.
\end{align*}
\end{proof}

\begin{remark}
Lemma \ref{lem:inductivestep} generalizes \cite[Ex. 4]{sempleandroth1949} from the setting of a generic surface in $\P^r$ with $r > 4$ to any surface in $\P^N$ with $N\geq4$ which arose as the image of an uncrumpled map. Moreover, the technique we use in this proof is an intersection theoretic computation, rather than the purely geometric argument used in the literature.
\end{remark}

We now turn our attention to the situation where $Y$ is ruled by lines.

\begin{lemma} 
\label{lem:projectivebundle} Suppose that $Y$ is ruled by lines. Then there is a morphism $\rho:X\to B$ realizing $X$ as a $\P^1$-bundle over a smooth curve $B$, and $\phi$ maps each fiber of $\rho$ isomorphically onto a line in $Y$. In this case, we write $g$ for the genus of $B$ and $F_b=\rho^{-1}(b)$ for the fiber over $b\in B$. 
\end{lemma}

\begin{proof} 
Since $Y$ is ruled by lines, let $q:S\to B$ be the universal family of an irreducible one-dimensional covering family, with $B$ normalized. Then $S$ is a $\P^1$-bundle, and the evaluation $e:S\to Y$ maps each fiber isomorphically onto a line. If $\deg e>1$, the off-diagonal part of $S\times_Y S$ maps dominantly to $B\times B$, since two distinct lines meet in at most one point. Thus two general members meet. Such a one-dimensional family is planar or consists of the generators of a cone; the former contradicts non-degeneracy, while the latter has evaluation degree $1$. Hence $e$ is birational. Since $S$ is normal and $X$ is the normalization of $Y$, it factors as $e=\phi\circ\mu$ for a birational morphism $\mu:S\to X$.

We claim that $\mu$ is an isomorphism. Suppose to the contrary that $\mu$ contracts some curve $C\subset S$. The curve $C$ cannot be contained in a fiber of $q$, because every fiber of $q$ maps to a line in $Y$, hence is not contracted by $e=\phi\circ\mu$. Therefore $C$ dominates $B$, so every general fiber $F_b=q^{-1}(b)$ meets $C$. Hence the curves $\mu(F_b)\subset X$ all pass through the point $\mu(C)\in X$. Since these lines cover $Y$, it follows that $Y$ is a cone with vertex $\phi(\mu(C))$. Since $X$ is the smooth normalization of $Y$, it follows that $Y\subset \P^N$ is an impossible 2-plane. 

This contradiction shows that $\mu$ is an isomorphism. Transporting the bundle map $q:S\to B$ across $\mu$ gives the desired $\P^1$-bundle $\rho:X\to B$, and the assertion about the fibers follows from the construction. 
\end{proof}

\begin{lemma}
    \label{lem:ruledcase}
    If $Y$ is ruled by lines, then $\gamma_2 \geq 2N-6,$ with equality holding precisely when $Y$ is a rational normal scroll.
\end{lemma}

\begin{proof} By Lemma \ref{lem:projectivebundle}, $\rho:X\to B$ is a $\P^1$-bundle over a smooth base curve $B$ of genus $g$. Let $\zeta$ represent the class of a section of $\mathcal O_X(1).$ The Picard group $\Pic(X)$ is isomorphic to $\rho^*\Pic(B)\oplus\Z\cdot\langle\zeta\rangle$. The Whitney sum formula applied to the relative Euler exact sequence on $X$ shows that the canonical class $K_X = -2\zeta + \rho^*D$ for some divisor class $D$ on $B$.  It follows that $2\zeta + K_X$ is pulled back from $B$ and so $(2\zeta + K_X)^2 = 0.$ Since $\zeta^2 = d$, again by the relative Euler sequence we find $c_2(T_X) = 4-4g.$ Applying Equation \eqref{equation:pinch},
$$\gamma_2 = 6\zeta^2 + 4\zeta K_X + K_X^2 - c_2(T_X)  = (2\zeta + K_X)^2 + 2\zeta^2 - c_2(T_X).$$
Substituting $(2\zeta+K_X)^2 = 0$, $\zeta^2 = d$, and $c_2(T_X) = 4-4g$ yields 
\begin{equation}
\label{eq:pinruled}
\gamma_2 = 2d + 4g - 4.
\end{equation}

Since $g\geq0$, we obtain $\gamma_2\geq2N-6$, with equality if and only if $d=N-1$ and $g=0$. In that case $Y$ is a variety of minimal degree and, since it is ruled by lines, a rational normal scroll. Varieties of minimal degree are normal, so the finite birational morphism $\phi:X\to Y$ is an isomorphism. Hence $Y$ is a smooth rational normal scroll.
\end{proof}

\begin{remark}
Lemma \ref{lem:ruledcase} is used below as an exit from the induction. There is a parallel inductive formulation in the ruled case: after blowing up a general point $x\in X$ and contracting the strict transform of the ruling fiber through $x$, one obtains the elementary transform $\bar X$, and the value of the pinch-point invariant associated to the induced map $\bar\phi:\bar X\to\P^{N-1}$ is $\gamma_2-2$.
If $\bar\phi$ were known to be uncrumpled, this would give a uniform induction, where the pinch point number drops by $4$ in the non-ruled case and by $2$ in the ruled case. We instead use the direct ruled computation above, since the finite-ramification assertion for $\bar\phi$ is not formal and is analogous to the uncrumpledness issue addressed in Theorem \ref{thm:uncrumpled}.
\end{remark}

We now complete the proof of Theorem \ref{thm:ppbound}.

\begin{proof}
We proceed by induction on $N\geq 4$. First suppose $N=4$. If $Y$ is ruled, then Lemma \ref{lem:ruledcase} gives $\gamma_2\geq 2$, with equality precisely when $Y$ is a rational normal scroll. If $Y$ is not ruled, then Theorem \ref{thm:uncrumpled} and Lemma \ref{lem:inductivestep} imply that the resolved inner projection $\wt{\phi}:\wt X\to\P^3$ is finitely ramified and satisfies $\wt\gamma_2=\gamma_2-4$. Since any finitely ramified map to $\P^3$ has $\wt\gamma_2\geq 0$, we get $\gamma_2\geq 4$. Thus the theorem holds for $N=4$.

Now suppose $N>4$. If $Y$ is ruled, then Lemma \ref{lem:ruledcase} gives $\gamma_2\geq 2N-6$, with equality precisely when $Y$ is a rational normal scroll. Assume from now on that $Y$ is not ruled. By Theorem \ref{thm:uncrumpled}, the resolved inner projection $\wt{\phi}:\wt X\to\P^{N-1}$ is uncrumpled, and by Lemma \ref{lem:inductivestep}, $\wt\gamma_2=\gamma_2-4$. Applying the induction hypothesis to $\wt{\phi}$ gives
$$\gamma_2 = \wt\gamma_2+4 \geq (2(N-1)-6)+4 > 2N-6.$$
Thus equality cannot occur in the non-ruled case, and holds for scrolls as in Lemma \ref{lem:ruledcase}.
\end{proof}

\section{Generalization}
\label{sec:generalization}
In this section, we wish to strengthen the classification result of Theorem \ref{thm:ppbound}. The length of the pinch point scheme is always even (this is a consequence of N\"oether's theorem and Equation \eqref{equation:pinch}). A general projection of the Veronese surface to $\P^3$ admits 6 pinch points; this is the smallest non-minimal number for a surface in $\P^5$. It is natural to wonder, are there other surfaces which are close to minimal with respect to the length of their pinch point scheme? 

Theorem \ref{thm:generalization} below gives a characterization of surfaces $Y = \phi(X)$, where $\phi:X\to\P^N$ is an uncrumpled map on a smooth surface $X$ and where $N$ is large enough, such that the difference between $\gamma_2$ and $2N-6$ is fixed. Before we state and prove that result, Lemma \ref{lem:ruledinnerproj} characterizes surfaces whose inner projection is ruled. Note that we continue to maintain the context of Setting \ref{innerprojectionsetting} in this section.

\begin{lemma}
\label{lem:ruledinnerproj}
If $\wt{\phi}(\wt{X})$ is ruled by lines, then either $Y$ is ruled by lines or $Y$ is the Veronese surface in $\P^5$, or $Y$ is a projection of the Veronese surface to $\P^4$.
\end{lemma}

\begin{proof} 
By Theorem \ref{thm:uncrumpled}, either $Y$ is ruled by lines or $\wt{\phi}$ is uncrumpled. Assume we are in the latter case. Since $\wt{\phi}(\wt X)$ is ruled, $\rho:\wt X\to B$ is a $\P^1$-bundle over a smooth curve by Lemma \ref{lem:projectivebundle}. The exceptional curve $E\subset \wt X$ is not a fiber of $\rho$, since $E^2=-1$ while every fiber has self-intersection $0$. Hence $\rho(E)=B$. Since $E\cong \P^1$, we have $B\cong \P^1$. Thus $\wt X$ is a Hirzebruch surface, and since it contains a $(-1)$-curve, $\wt X\cong \mathbb F_1$. Contracting $E$ gives $X\cong \P^2$. 

Under this identification, the ruling on $\mathbb F_1$ is the pencil of lines in $\P^2$ through $x$. Let $H$ denote the pullback to $\wt X\cong\mathbb F_1$ of a line in $X\cong\P^2$, and write $\zeta=mH$. Writing $\zeta=mH$, a ruling fiber has class $H-E$ and degree $(mH-E)\cdot(H-E)=m-1=1$, so $m=2$. Hence $m=2$. Since $\wt{\phi}$ maps the fibers of this ruling to lines, the map $\phi:X\cong\P^2\to\P^N$ is induced by a linear system of conics. Therefore $N\in\{4,5\}$. If $N=5$, the system is complete, and $Y\subset\P^5$ is the Veronese surface. If $N=4$, then $Y\subset\P^4$ is a projection of the Veronese surface. 
\end{proof}

\begin{theorem}
\label{thm:generalization}
For any positive integer $i$, if $N\geq 3+i\geq 4$, then $\gamma_2 = 2N-6+2i$ if and only if
\begin{enumerate}
\item $Y$ is ruled by lines and $\deg Y = N-1+i-2g$; in particular, $g\leq \frac{i}{2}$, or
\item $N = 5$, $i = 1$, and $Y$ is the Veronese surface, or
\item $4\leq N\leq 9$, $i = N - 3$, and $Y$ is a del Pezzo surface of degree $N$.
\end{enumerate}
\end{theorem}

\begin{proof}
Suppose $\gamma_2 = 2N-6+2i.$ If $Y$ is ruled by lines, then by Lemma \ref{lem:ruledcase}, $\gamma_2=2d+4g-4.$ Comparing with $\gamma_2=2N-6+2i$ gives $d=N-1+i-2g$, and since $d\geq N-1$, it follows that $g\leq \frac{i}{2}$. 

Assume from now on that $Y$ is not ruled. By Theorem \ref{thm:uncrumpled}, $\wt{\phi}$ is uncrumpled, and 
\begin{equation}
\label{eq:generalizedinduction}
\wt\gamma_2 = \gamma_2 - 4 = 2N - 10 + 2i = 2(N-1) - 6 + 2(i-1)
\end{equation}
by Lemma \ref{lem:inductivestep}. We now have two cases.

\textbf{Case I.} Suppose $N\geq i+4$. We argue by induction on $i$, with the case $i=0$ given by Theorem \ref{thm:ppbound}. By Equation \eqref{eq:generalizedinduction}, the induction hypothesis with parameter $i-1$ implies that $\wt\phi(\wt X)$ is either ruled by lines, the Veronese surface in $\P^5$, or a del Pezzo surface of degree $N-1$. The del Pezzo case would require $i=N-3$, contrary to $N\geq i+4$, while the Veronese case would imply $\wt X\cong\P^2$, contradicting the existence of the exceptional $(-1)$-curve $E$. Thus $\wt\phi(\wt X)$ is ruled by lines. By Lemma \ref{lem:ruledinnerproj}, and since $Y$ is not ruled, $Y$ is either the Veronese surface in $\P^5$ or a projection of the Veronese surface to $\P^4$. The latter is excluded by $N\geq i+4\geq5$, while the former forces $N=5$ and $i=1$.

\textbf{Case II.} Suppose $N = 3 + i$, with $i\geq 1$. If $i = 1$, then $N = 4$, and therefore $\gamma_2 = 4$. By Equation \eqref{eq:generalizedinduction}, $\wt\gamma_2 = 0$, so $\wt{\phi}$ is an unramified map from a smooth surface to $\P^3$. By Theorem \ref{thm:fulhanconsequence}, $\wt{X}$ embeds into $\P^3$. Moreover, $\wt{\phi}(\wt{X})$ contains a line whose self intersection is $-1$, so by adjunction it must be a smooth cubic surface (a degree $3$ del Pezzo surface). Thus $\wt\phi(\wt X)$ is a del Pezzo surface of degree $N-1$.

Assume now that $i>1$. By induction, $\wt\phi(\wt X)$ is ruled by lines, the Veronese surface, or a del Pezzo surface of degree $N-1$. The Veronese case contradicts the exceptional $(-1)$-curve $E$. In the ruled case, Lemma \ref{lem:ruledinnerproj} forces $Y$ to be the Veronese surface in $\P^5$ or its projection to $\P^4$. The latter is excluded by $N\geq5$, while the former forces $(N,i)=(5,2)$, but then $\gamma_2=6\neq8$. Hence $\wt\phi(\wt X)$ is a del Pezzo surface of degree $N-1$.

Since $\wt\zeta=-K_{\wt X}$, the identities $\wt\zeta=\beta^*\zeta-E$ and $K_{\wt X}=\beta^*K_X+E$ give $\zeta=-K_X$. Moreover, $\zeta^2=\wt\zeta^2+1=N$. Thus $X$ is a del Pezzo surface of degree $N$. Since $h^0(X,-K_X)=N+1$, the non-degenerate map $\phi:X\to\P^N$ is the complete anticanonical embedding. Hence $Y$ is a del Pezzo surface of degree $N$, with $4\leq N\leq9$.

Conversely, in the ruled case the equality follows immediately from Lemma \ref{lem:ruledcase} and the relation $d = N-1+i-2g$. For the Veronese surface in $\P^5$, the pinch-point formula gives $\gamma_2=6=2\cdot 5-6+2$. If $Y\subset \P^N$ is a del Pezzo surface of degree $N$, then $\zeta=-K_Y$, $\zeta^2=N$, and $c_2(T_Y)=12-N$, so by Equation \ref{equation:pinch},
$\gamma_2=2N-6+2(N-3).$
\end{proof}

\section{Projective characters and further questions}
\label{sec:futurework}

Let $X^n\subset\P^N$ be a smooth non-degenerate projective variety, and let $\mathfrak g:X\to\G(n,N)$ be its Gauss map. For each partition $\lambda\vdash n$, let $\sigma_\lambda^*$ denote the corresponding dual Schubert class. We may write
$$\mathfrak g_*[X]=\sum_{\lambda\vdash n}\gamma_\lambda\sigma_\lambda^*.$$
The integers $\gamma_\lambda$ are the projective characters of $X$. They measure different aspects of the variation of the embedded tangent spaces of $X$, and each gives rise to a natural extremal problem: determine the sharp lower bound for $\gamma_\lambda$ and classify the varieties attaining it.

For surfaces, the two projective characters are $\gamma_2$ and $\gamma_{1,1}$. The invariant studied in this paper is $\gamma_2$, while $\gamma_{1,1}$ is the class of the surface. Thus Theorem \ref{theorem:simplified} determines the sharp lower bound for the one-row character $\gamma_2$. At the opposite extreme, the one-column character $\gamma_{\mathbf 1^n}$ is the top polar degree; its vanishing is the classical dual-defect problem studied by Griffiths--Harris, Ein, and others.

The natural higher-dimensional analogue of the present paper concerns the one-row character $\gamma_n$. Write $\Theta_X\subset\P^N$ for its tangent variety. When $\dim\Theta_X=2n$, this character is closely related to the ramification scheme of a general projection of $X$ to $\P^{2n-1}$ and satisfies $\gamma_n=\deg(\eta)\deg\Theta_X,$ where $\eta$ is the natural map from the abstract tangential variety to $\Theta_X$ as in \cite{gomez2026tangentdegreedegreetangent}. A direct computation gives $\gamma_n=2(N-2n+1)$ for every rational normal scroll of dimension $n$. Equality is not confined to scrolls in higher dimensions: for example, $\P^1\times Q^{n-1}\subset\P^{2n+1}$ also satisfies $\gamma_n=4$.

\begin{problem}
\label{problem:projectivecharacters}
Let $X^n\subset\P^N$ be a smooth non-degenerate variety with $\dim\Theta_X=2n$. Determine the sharp lower bound for $\gamma_n$ and classify the varieties attaining it.
\end{problem}

There is also a variational formulation of the problem. A general projection of $X$ to $\P^{2n-1}$ is parametrized by the Grassmannian $\G(N-2n,N)$, which has dimension $2n(N-2n+1)$. If the corresponding ramification scheme is reduced of length $p$, then it determines a point of the symmetric product $X^p/S_p$, which has dimension $np$. Consequently, if first order deformations of the projection always induced non-trivial first order deformations of the ramification cycle, then one would obtain $p\geq 2(N-2n+1)$. This would give another possible proof of the results of \cite{gomez2026tangentdegreedegreetangent} recalled above.

Moving to the enumerative world, the calculation that $\gamma_n=2(N-2n+1)$ for a rational normal scroll $X\subset\P^N$ opens up a collection of challenging enumerative problems. By keeping track of the ramification scheme of the projection, we obtain the projection-ramification map
$$\Phi:\G(N-2n,N)\dashrightarrow X^{2(N-2n+1)}/S_{2(N-2n+1)},$$
a rational map between two schemes of the same dimension. And so we pose:

\begin{problem}
\label{prob:hilbdegree}
For each rational normal scroll $S(a_1,\dots,a_n)$, determine the degree of the projection-ramification map $\Phi$.
\end{problem}

Problem \ref{prob:hilbdegree} is already interesting for surfaces. For the scroll $S(1,2)$ in $\P^4$, $\deg\Phi=1$, because two generic points on the cubic scroll $S(1,2)$ uniquely determine a point of projection: one simply intersects the two tangent planes of $S(1,2)$ at those points. In $\P^5$, there is a material difference between the two quartic scrolls $S(2,2)$ and $S(1,3)$. In general, the indeterminacy locus of $\Phi$ can be complicated, and therefore makes the computation of the degree of $\Phi$ challenging. The directrix contribution can become more complicated as the scroll becomes more imbalanced, although the family $S(1,b)$ admits the Wronski reduction described below. Problems \ref{problem:projectivecharacters} and \ref{prob:hilbdegree} should be compared to the study undertaken in \cite{deopurkar2020ramification}.

\begin{remark}
When $X$ is a rational normal curve, the map $\Phi$ is a morphism, called the Wronski map, and its degree is the Catalan number 
$$\deg \Phi = \frac{1}{N}\binom{2N-2}{N-1};$$ 
see \cite{eremenko2002rational}. Similarly, for the scrolls $S(1,b)\subset \P^{b+2}$, quotienting by the directrix line reduces Problem \ref{prob:hilbdegree} to the Wronski map for the rational normal curve of degree $b$. Consequently,
$$\deg\Phi_{S(1,b)}=\frac{1}{b}\binom{2b-2}{b-1}.$$
\end{remark}

\section*{Acknowledgements}

We are very grateful to Fyodor Zak and Ragni Piene for their detailed comments, alternative perspectives, and encouragement. We also thank Ciro Ciliberto and Antonio Lanteri for their helpful insights. We are grateful to C\'esar Lozano-Huerta, Emilia Mezzetti, Flaminio Flamini, Igor Dolgachev, and Izzet Coskun for their encouragement, comments, suggestions and questions. We thank the anonymous reviewer for many detailed insights and recommendations. 


\vspace{9pt}

\textbf{Adam Cartisano} 

Department of Mathematics, Computer Science and Data Science

Belmont University, Nashville, TN, USA 

Email: adam.cartisano@belmont.edu 

\vspace{9pt}

\textbf{Anand Patel} 

Department of Mathematics

Oklahoma State University, Stillwater, OK, USA 

Email: anand.patel@okstate.edu

\end{document}